\documentclass[reqno]{amsart}
\usepackage{amsmath}
\usepackage{amscd}
\usepackage{amsfonts}
\usepackage{amsthm}
\usepackage{amssymb}
\usepackage{graphicx}
\usepackage[all]{xy}
\usepackage{color}
\usepackage{comment}
\usepackage{fancybox}
\usepackage{epsf}
\usepackage[abs]{overpic}
\usepackage{layout}

\graphicspath{{./Figures/}}

\numberwithin{equation}{subsection}

\theoremstyle{plain}
\newtheorem{theorem}{Theorem}[section]

\newtheorem{proposition}[theorem]{Proposition}

\theoremstyle{definition}

\theoremstyle{remark}
\newtheorem{remark}[theorem]{Remark}

\newtheorem*{acknowledgments}{Acknowledgments}

\newcommand\nc{\newcommand}
\nc\rnc{\renewcommand}

\nc\block[2]{\begin{#1}#2\end{#1}}
\nc\comm[1]{\ {\tt[* #1 *]}\ }
\nc\note[1]{{\small{{\tt [note:}\ {#1} {]}}}}

\nc\hide[1]{--}

\nc\nempty{\neq\emptyset}

\nc\xto[1]{{\overset{#1}{\longrightarrow}}}
\nc\yto[1]{{\underset{#1}{\longrightarrow}}}
\nc\xyto[2]{{\overset{#1}{\underset{#2}{\longrightarrow}}}}
\nc\opn{\operatorname}
\nc\LCS{\opn{LCS}}
\nc\Lie{\opn{Lie}}
\nc\id{\opn{id}}
\nc\sgn{\opn{sgn}}
\nc\Center{\opn{Center}}
\nc\ad{{\opn{ad}}}
\nc\adr{\opn{ad}^r}
\nc\coad{\opn{coad}}
\nc\coadr{\opn{coad}^r}
\nc\op{{\opn{op}}}
\nc\End{\opn{End}}
\nc\Span{\opn{Span}}
\nc\Spec{\opn{Spec}}
\nc\even{{\opn{even}}}
\nc\odd{{\opn{odd}}}
\nc\Gal{\opn{Gal}}
\nc\Aut{\opn{Aut}}
\nc\Ob{\opn{Ob}}
\nc\ob{\opn{ob}}
\nc\Mor{\opn{Mor}}
\nc\opint{\opn{int}}
\nc\Nbd{\opn{Nbd}}

\nc\ev{{\opn{ev}}}
\nc\coev{{\opn{coev}}}
\rnc\ev{d}
\rnc\coev{b}

\nc\ol{\overline}
\nc\half{{\frac12}}
\nc\halfof[1]{{\frac{#1}2}}
\nc\bb[2]{\biggl[\begin{matrix}{#1}\\{#2}\end{matrix}\biggr]}
\nc\onto\twoheadrightarrow
\nc\projto{\underset{\text{proj}}{\longrightarrow}}
\nc\no[1]{}
\nc\ok{\comm{ok?}}
\nc\ho{{\hat\otimes }}
\nc\zzzcolon {\colon\thinspace}
\nc\zzzsharp {\sharp}
\nc\plim{\varprojlim}
\nc\np{\newpage}
\nc\modone {{\mathbf 1}}
\nc\g{{\mathfrak g}}
\nc\modk {{\mathbf k}}
\nc\modL {{\mathcal L}}
\nc\bbN{{\mathbb N}}
\nc\modQ {{\mathbb Q}}
\nc\modR {{\mathcal R}}
\nc\R{{\mathbb R}}
\nc\modZ {{\mathbb Z}}
\nc\ul{\underline}

\nc\simeqto{\overset{\simeq}{\rightarrow }}
\nc\trl{\triangleleft}
\nc\trr{\triangleright}

\nc\fig[2]{\begin{figure}
\Small
    \boxed{\tt #1}
    \caption{#2}
    \label{#1}
    \end{figure}}
\nc\FIGURE[2]{\begin{figure}
    \boxed{\tt fig: #1}
    \caption{#2}
    \label{fig:#1}
    \end{figure}}

\nc\xysquare[8]{\xymatrix{
    #1 \ar[r]#5 \ar[d]#6 & #2 \ar[d]#7 \\
    #3 \ar[r]#8          & #4
  }
  }

\nc\modC {{\mathcal C}}
\nc\modE {{\mathcal E}}
\nc\modV {{\mathcal V}}
\nc\Vect{\mathbf{Vect}}
\nc\Vectfin{\Vect^{\mathrm{fin}}}
\nc\Sets{\mathbf{Sets}}
\nc\Spaces{\mathbf{Spaces}}
\nc\Mod{\mathbf{Mod}}
\nc\Modfin{\Mod^{\mathrm{fin}}}
\nc\Cat{\mathbf{Cat}}

\nc\ct{\overset{\cong}{\to}}
\nc\bu{\bullet}
\nc\sq{\simeq}

\nc\blue[1]{{\textcolor[rgb]{0,0,.9}{#1}}}
\nc\bluen[1]{\blue{[[#1]]}}
\nc\bnote{\bluen}
\nc\red[1]{{\textcolor[rgb]{.9,0,0}{#1}}}
\nc\redn[1]{\red{[[#1]]}}

\nc\modP {\mathcal{P}}
\nc\LB{\mathbf{LB}}
\nc\modI {\mathcal{I}}
\nc\mt{\mapsto}
\nc\modB {\mathcal{B}}
\nc\gr{\opn{gr}}
\nc\modT {\mathcal{T}}
\nc\cv{{\opn{conv}}}
\nc\modH {\mathbf{H}}
\nc\modM {\mathcal{M}}

\nc\ci{\circ}
\nc\ti{\tilde}
\nc\ee[2]{\eta _{#1}\epsilon _{b#2}}
\nc\sh{\mathrm{sh}}

\nc\move[1]{\underset{#1}{\Longrightarrow}}

\nc\incl{\mathrm{incl.}}
\nc\modr {\mathbf{r}}
\nc\C{\mathbb{C}}

\nc\Id{\mathrm{Id}}
\nc\Hom{\opn{Hom}}
\rnc\xto{\xrightarrow}
\nc\cut[1]{{\color{red}[[#1]]}}
\nc\ModH{\Mod_H}
\nc\Piv{\mathcal{P}}
\nc\Rib{\opn{Rib}}
\nc\figu{\mathrm{fig}}
\nc\RigCat{\mathbf{RigCat}}
\nc\PivCat{\mathbf{PivCat}}

\nc\hhrule{\medskip\hrule\medskip}

\nc\XX[1]{(#1,#1^*,\ev_{#1},\coev_{#1},\ev_{#1^*},\coev_{#1^*})}
\nc\XXs[1]{(#1^*,#1,\ev_{#1^*},\coev_{#1^*},\ev_{#1},\coev_{#1})}
\nc\evv[1]{\ev_#1,\coev_#1,\ev_{#1^*},\coev_{#1^*}}

\nc\modp {\mathrm{p}}
\nc\Set{\mathbf{Set}}
\nc\Zp{\modZ _{\ge 0}}
\nc\modm {\mathbf{m}}
\nc\modn {\mathbf{n}}
\nc\modl {\mathbf{l}}
\nc\modf {\mathbf{f}}
\nc\mods {\mathbf{s}}
\nc\modd {\mathbf{d}}
\nc\modN {\mathcal{N}}
\nc\modW {\mathcal{W}}

\nc\hr{\medskip\hrule\medskip}

\nc\YD{\mathcal{YD}}
\nc\rYD{r\YD}
\nc\pYD{p\YD}
\nc\YDd{\YD^{\mathrm{d}}}
\nc\YDr{\YD^{\mathrm{r}}}
\nc\YDp{\YD^{\mathrm{p}}}
\nc\YDH{\YD_H}
\nc\rYDH{\rYD_H}
\nc\pYDH{\pYD_H}
\nc\YDHVp{(\YDH^\modV )^p}
\nc\YDHV{\YDH^\modV }
\nc\pYDHV{\pYDH^\modV }
\nc\rYDHV{\rYDH^\modV }

\nc\modD {\mathcal{D}}
\nc\X{\mathbf{X}}

\nc\bfv{\mathbf{v}}
\nc\ot{\otimes}
\newcommand{\Qh}{\mathbb{Q}[[h]]}

\begin{document}
\title[Ribbon Yetter--Drinfeld modules and tangle invariants]{Ribbon Yetter--Drinfeld modules\\and tangle invariants}

\date{April 6, 2022.}
\author{Kazuo Habiro}
\author{Yuka Kotorii}

\address{Department of Mathematics, Kyoto University, Kitashirakawa-Oiwakecho, Sakyo-ku, Kyoto 606-8502, Japan}
\email{habiro@math.kyoto-u.ac.jp}

\address{Department of Mathematics, Graduate school of Science, Hiroshima University, 1-7-1 Kagamiyama Higashi-hiroshima City, Hiroshima 739-8521 Japan}
\address{Interdisciplinary Theoretical and Mathematical Sciences Program, RIKEN, 2-1, Hirosawa, Wako, Saitama 351-0198, Japan}
\email{kotorii@hiroshima-u.ac.jp}

\maketitle
\begin{abstract}
    We define notions of pivotal and ribbon objects in a monoidal category.  These constructions give pivotal or ribbon monoidal categories from a monoidal category which is not necessarily with duals.
    We apply this construction to the braided monoidal category of Yetter--Drinfeld modules over a Hopf algebra.  This gives rise to the notion of ribbon Yetter--Drinfeld modules over a Hopf algebra, which form ribbon categories.  This gives an invariant of tangles.
\end{abstract}

\tableofcontents


\section{Introduction}

Reshetikhin and Turaev \cite{RT} introduced the notion of ribbon Hopf algebra and showed that the category of finite-dimensional modules over a ribbon Hopf algebra has a ribbon category structure.
Since the category $\mathcal T$ of framed, oriented tangles is a free ribbon category generated by one object \cite{Shum,Turaev}, a ribbon category yields a functor from the tangle category $\mathcal T$ to the category of finite-dimensional vector spaces, and thus gives a functorial invariants of tangles.
Many quantum link invariants, such as the Jones polynomial, can be constructed in this way.

Yetter and Drinfeld \cite{Yetter,Drinfeld1} introduced the notion of Yetter--Drinfeld modules over a Hopf algebra.  It is both a module and a comodule over the Hopf algebra such that the action and the coaction satisfy a certain compatibility condition.
The category $\YD_H$ of Yetter--Drinfeld modules over $H$ has a structure of a braided category, and its full subcategory $\YD_H^{fd}$ of finite-dimensional Yetter--Drinfeld modules has a structure of a rigid (i.e., with left duals) braided category \cite{AG}.
This is close to, but different from, having a ribbon category and tangle invariants.
It is well known that for a finite-dimensional Hopf algebra $H$, the category $\YD_H$ of Yetter--Drinfeld modules over $H$ is equivalent to the category of $D(H)$-module.
Here the Drinfeld double $D(H)$ of $H$ is a quasi-triangular Hopf algebra with underlying space $H\otimes H^*$.
It is also well known that one can adjoin a ribbon element $v$ to a quasi-triangular Hopf algebra $H$ to obtain a ribbon Hopf algebra $H\oplus Hv$  \cite{RT}.
Thus, starting from a finite-dimensional Hopf algebra $H$, one obtains a ribbon category $D(H)\oplus D(H)v\text{-mod}^{fd}$ and the associated tangle invariant.

In this paper, we take a different approach of getting a ribbon category from Yetter--Drinfeld modules.
We define a notion of \emph{ribbon Yetter--Drinfeld module} over a Hopf algebra, which is a finite-dimensional Yetter--Drinfeld module $X$ over $H$ equipped with a linear automorphism $\gamma_X : X\to X$ that plays a role of right curl (see \eqref{e22}).
We prove the following result.

\begin{theorem}[Proposition \ref{r24} and Theorem \ref{r34}]
  Let $H$ be a Hopf algebra with invertible antipode over a field $k$.
  Then the category $\rYD_H\text{-mod}$ of ribbon Yetter--Drinfeld modules admits a
  structure of a ribbon category, and hence gives rise to tangle invariants.  
\end{theorem}
Note that in this theorem we do not assume that $H$ is finite-dimensional.

The notions of Hopf algebras and Yetter--Drinfeld modules can be defined in any symmetric monoidal category.  In this paper, we work in this general setting.

Here we outline the rest of the paper.
In Section 2, we recall some notions such as braided monoidal categories.
In Section 3, we introduce the notion of pivotal objects and ribbon objects in a monoidal category, which does not necessarily have duals.
The category of pivotal objects in a braided category has a structure of a pivotal, braided category, and similarly
the category of ribbon objects in a braided category has a structure of a ribbon category.
In Section 4, we first recall the notion of Yetter--Drinfeld module over a Hopf algebra, and then we define a ribbon Yetter--Drinfeld module to be a ribbon objects in the braided category of Yetter--Drinfeld modules.
In Section 5, we give some examples.

\begin{acknowledgments}
K.H. is supported by JSPS KAKENHI Grant Number 18H01119.
Y.K. is supported by JSPS KAKENHI, Grant-in-Aid for Early-Career Scientists Grant Number 20K14322, and by RIKEN iTHEMS Program.
\end{acknowledgments}

\section{Preliminaries}

In this section, we recall the notions of (strict) monoidal categories, braided
monoidal categories, symmetric monoidal categories, pivotal categories
and ribbon categories.  See  \cite{Kassel,MacLane,Selinger} for more details.

\subsection{Monoidal categories}

A {\em monoidal category} is a category $\modM $ equipped with a functor
$\otimes \zzzcolon \modM \times \modM \rightarrow \modM $ which is associative and unital with respect to an
object $I$ up to natural isomorphisms satisfying the well-known
coherence conditions.  If the associativity and unitality strictly
hold, then $\modM $ is called a {\em strict monoidal category}.  In this
paper, we will mainly work with strict monoidal categories for
simplicity.  It is straightforward
to
generalize the results below to non-strict monoidal categories.

The opposite $\modM ^\op$ of a strict monoidal category $\modM $ has a strict
monoidal structure $\otimes ^\op$, where $X\otimes ^{\op}Y= Y\otimes X$ for objects $X$
and $Y$ in $\modM $ and $f\otimes ^{\op}g=g\otimes f$ for morphisms $f$ and $g$ in
$\modM $.

A {\em strict monoidal functor} $F\zzzcolon \modM \rightarrow \modN $ between strict monoidal
categories $\modM $ and $\modN $ is a functor $F\zzzcolon \modM \rightarrow \modN $ such that $F(I)=I$ and
$F\circ\otimes =\otimes \circ(F\times F)\zzzcolon \modM \times \modM \rightarrow \modN $.

We will often use the graphical notation or string diagrams for
morphisms in strict monoidal categories.
Each object may be depicted by a vertical string. A morphism $f\zzzcolon x\rightarrow y$
is depicted as in Figure \ref{fig001}.
\begin{figure}[t] 
$$
{\raisebox{-7.8mm}{\begin{overpic}[width=8.5mm]{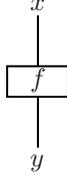}\put(9,-6){$y$}\put(9,53){$x$}\put(9,24){$f$}\end{overpic}}}
$$
\caption{String diagram.} \label{fig001}
\end{figure}
The composition is given by vertical pasting of diagrams, and the
tensor product is given by horizontal juxtaposition, see Figure \ref{fig002}.
\begin{figure}[t] 
$$
{\raisebox{-7.8mm}{\begin{overpic}[width=8.5mm]{morphism.eps}\put(9,-7){$z$}\put(9,54){$x$}\put(2,24){$g \circ f$}\end{overpic}}}  \hspace{.5cm} = \hspace{.5cm} 
{\raisebox{-7.8mm}{\begin{overpic}[width=8.5mm]{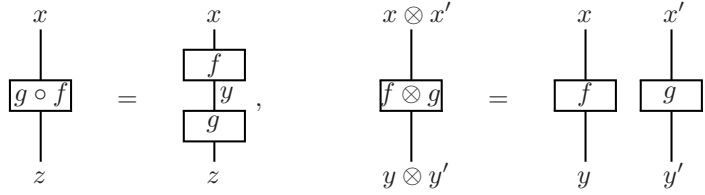}\put(9,-7){$z$}\put(9,54){$x$}\put(9,35){$f$}\put(14,25){$y$}\put(9,14){$g$}\end{overpic}}} \ , \hspace{1.5cm}
{\raisebox{-7.8mm}{\begin{overpic}[width=8.5mm]{morphism.eps}\put(1,-7){$y \otimes y'$}\put(1,54){$x \otimes x'$}\put(0,24){$f \otimes g$}\end{overpic}}} \hspace{.5cm} = \hspace{.5cm} 
{\raisebox{-7.8mm}{\begin{overpic}[width=8.5mm]{morphism.eps}\put(9,-7){$y$}\put(9,54){$x$}\put(9,24){$f$}\end{overpic}}} \hspace{.3cm}
{\raisebox{-7.8mm}{\begin{overpic}[width=8.5mm]{morphism.eps}\put(9,-7){$y'$}\put(9,54){$x'$}\put(9,25){$g$}\end{overpic}}}
$$
\caption{Composition and tensor product.} \label{fig002}
\end{figure}

\subsection{Duality}

Let $\modM $ be a strict monoidal category.  A {\em duality} in $\modM $ is a
quadruple $(X,X^*,\ev_X,\coev_X)$ of objects $X$ and $X^*$ and
morphisms
\begin{gather*}
  \ev_X= \ {\raisebox{-4.mm}{\begin{overpic}[width=10mm]{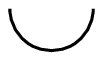}\put(-4,20){$X^*$}\put(23,20){$X$}\end{overpic}}} \ \ \zzzcolon X^*\otimes X\rightarrow I,\quad \coev_X= \ {\raisebox{.mm}{\begin{overpic}[width=10mm, angle=180]{bottom.eps}\put(-4,-9){$X$}\put(23,-9){$X^*$}\end{overpic}}} \ \ \ \zzzcolon I\rightarrow X\otimes X^*
\end{gather*}
in $\modM $ such that\\
\begin{gather*}
{\raisebox{-6.mm}{\begin{overpic}[width=15mm]{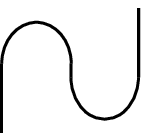}\put(38,42){$X$}\put(-4,-7){$X$}\end{overpic}}} \hspace{.5cm} =  \hspace{.5cm}
{\raisebox{-6.mm}{\begin{overpic}[width=.4mm]{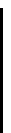}\put(-4,44){$X$}\put(-4,-7){$X$}\end{overpic}}}\hspace{.4cm} , \hspace{1.4cm}
{\raisebox{-6.mm}{\begin{overpic}[width=15mm]{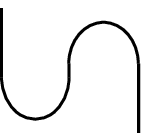}\put(36,-7){$X^*$}\put(-4,42){$X^*$}\end{overpic}}}  \hspace{.5cm} =  \hspace{.5cm}
{\raisebox{-6.mm}{\begin{overpic}[width=.4mm]{isotopy1.eps}\put(-4,44){$X^*$}\put(-4,-7){$X^*$}\end{overpic}}} \hspace{.4cm}  . \\
\end{gather*}
The object $X^*$ and also the triple $(X^*,\ev_X,\coev_X)$ are called
a {\em left dual} of $X$.

A strict monoidal category $\modM $ is called {\em $($left$)$ rigid} if each
object $X$ has a left dual.  Choosing a left dual of each object gives
rise to a functor
\begin{gather*}
  (-)^*\zzzcolon \modM ^\op \rightarrow \modM ,
\end{gather*}
which maps each object $X$ to its chosen left dual $X^*$, and each
morphism $f\zzzcolon X\rightarrow Y$ to its left dual $f^*\zzzcolon Y^*\rightarrow X^*$ defined by \\
\begin{gather}
  \label{e10}
         f^* = \ {\raisebox{-7.8mm}{\begin{overpic}[width=15mm]{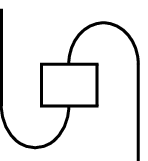}\put(37,-8){$X^*$}\put(-2,50){$Y^*$}\put(18,21.5){$f$}\end{overpic}}} \hspace{.4cm}  .
\end{gather}
\\

\subsection{Strict pivotal categories}

A {\em strict pivotal category} is a rigid strict monoidal category
$\modM $ with chosen left duals (and hence the left duality functor
$(-)^*\zzzcolon \modM ^\op \rightarrow \modM $) such that
\begin{enumerate}
\item $(-)^*\zzzcolon \modM ^\op\rightarrow \modM $ is a strict monoidal functor,
\item $\ev_I=\coev_I=\id_I$,
\item $\ev_{X\otimes Y}=\ev_Y(Y^*\otimes \ev_X\otimes Y)$ and
  $\coev_{X\otimes Y}=(X\otimes \coev_Y\otimes X^*)\coev_X$ for any objects $X$ and $Y$
  in $\modM $,
\item $(-)^{**}=\Id_\modM \zzzcolon \modM \rightarrow \modM $.
\end{enumerate}

A {\em strict pivotal functor}
  $F\zzzcolon \modM \rightarrow \modN $ between strict pivotal categories is a strict monoidal
  functor such that we have
\begin{gather*}
  F(X^*)=F(X)^*,\quad
  F(\ev_X)=\ev_{F(X)},\quad
  F(\coev_X)=\ev_{F(X)},
\end{gather*}
for all objects $X$ in $\modM $.   Such $F$ commutes
with the left duality functors, i.e., we have
\begin{gather*}
F((-)^*)=F(-)^* : \modM^\op\to \modN
\end{gather*}

\subsection{Braided categories}

A strict monoidal category $\modM $ is {\em braided} if it is equipped
with a natural family of isomorphisms, called the \emph{braidings}, \\
\begin{gather*}
  \psi _{X,Y}= \quad {\raisebox{-5mm}{\begin{overpic}[width=10mm]{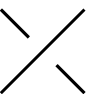}\put(-4,32){$X$}\put(25,32){$Y$}\put(-4,-8){$Y$}\put(23,-8){$X$}\end{overpic}}} \quad \zzzcolon X\otimes Y\overset{\simeq}{\longrightarrow} Y\otimes X \\
\end{gather*}
for all $X,Y\in \Ob(\modM )$, such that
\begin{gather*}
  \psi _{X\otimes Y,Z}=(\psi _{X,Z}\otimes Y)(X\otimes \psi _{Y,Z}),\quad
  \psi _{X,Y\otimes Z}=(Y\otimes \psi _{X,Z})(\psi _{X,Y}\otimes Z)
\end{gather*}
for $X,Y,Z\in \Ob(\modM )$.

A {\em strict symmetric monoidal category} is a strict braided 
monoidal category $\modM $ such that
$\psi _{Y,X}\psi _{X,Y}=\id_{X\otimes Y}\zzzcolon X\otimes Y\rightarrow X\otimes Y$ for all $X,Y\in \Ob(\modM )$.  The
braidings in this case are called the {\em symmetries}, which will be
denoted by \\
\begin{gather*}
  P_{X,Y}= \quad {\raisebox{-5mm}{\begin{overpic}[width=10mm]{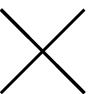}\put(-4,32){$X$}\put(25,32){$Y$}\put(-4,-8){$Y$}\put(23,-8){$X$}\end{overpic}}} \quad  \zzzcolon X\otimes Y\rightarrow Y\otimes X \\
\end{gather*}
instead of $\psi _{X,Y}$.

A strict monoidal functor $F\zzzcolon \modM \rightarrow \modN $ between braided strict monoidal
categories $\modM $ and $\modN $ is called a {\em strict braided monoidal
  functor} if we have
\begin{gather*}
  F(\psi _{X,Y})=\psi _{F(X),F(Y)}
\end{gather*}
for all objects $X$ and $Y$ in $\modM $.  If both $\modM $ and $\modN $ are
symmetric, then $F$ is called a {\em strict symmetric monoidal functor}.

\subsection{Strict braided pivotal categories and strict ribbon categories}

In a strict braided pivotal category $\modM $, there are natural
isomorphisms
\begin{gather*}
  c^R_X\zzzcolon X\rightarrow X,\quad
  c^L_X\zzzcolon X\rightarrow X
\end{gather*}
for objects $X$ in $\modM $ defined by
\medskip
\begin{gather}
  \label{e23}
  c^R_X = \ {\raisebox{-7.mm}{\begin{overpic}[width=10mm]{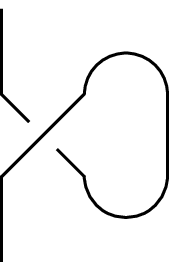}\put(-4,45){$X$}\put(-4,-9){$X$}\end{overpic}}} \ ,\quad\quad
    c^L_X = \ {\raisebox{-7.mm}{\begin{overpic}[width=10mm, angle=180]{rightpositivecurl.eps}\put(24,45){$X$}\put(24,-9){$X$}\end{overpic}}} \ \ . 
\end{gather} \\
We call $c^R_X$ (resp. $c^L_X$) the {\em right} (resp. {\em left})
{\em positive curl} for $X$.
The inverses of $c^R_X$ and $c^L_X$ are given by
\begin{gather*}
 {\raisebox{-7.mm}{\begin{overpic}[width=10mm]{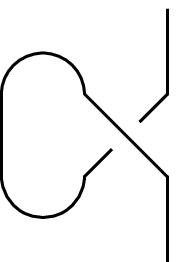}\put(24,45){$X$}\put(24,-9){$X$}\end{overpic}}}  \hspace{.8cm} and \hspace{.8cm} {\raisebox{-7.mm}{\begin{overpic}[width=10mm, angle=180]{rightpositivecurlinverse.eps}\put(-4,45){$X$}\put(-4,-9){$X$}\end{overpic}}} \ .
\end{gather*} \\

A {\em strict ribbon category} is a strict braided pivotal category
$\modM $ such that the right and left positive curls coincide, i.e.,
\begin{gather}
  \label{e11} c^R_X=c^L_X
\end{gather}
for any object $X$ in $\modM $.  In this case, we write $c_X=c^R_X=c^L_X$.

\begin{remark}
  \label{r38}
  In the literature, a ribbon (or tortile) category is usually defined
  to be a rigid braided monoidal category with a twist
  \cite{Shum, RT, Kassel}.  Here a
  {\em twist} for a strict braided monoidal category $\modM $ is a natural
  family of isomorphisms $\theta _X\zzzcolon X\rightarrow X$ such that we have 
  and
  \begin{gather*}
  \theta _I=\id_I,\quad
    \theta _{X\otimes Y} = \psi _{Y,X}\psi _{X,Y}(\theta _X\otimes \theta _Y),\quad
    \theta _{X^*} = (\theta _X)^*
  \end{gather*}
  for all objects $X$ and $Y$ in $\modM $.  The above definition of strict
  ribbon category is equivalent to the usual one.
\end{remark}

A {\em strict ribbon functor} $F\zzzcolon \modM \rightarrow \modN $ between two strict ribbon
categories $\modM $ and $\modN $ is just a {\em strict braided pivotal functor}, i.e.,
a strict pivotal functor $F\zzzcolon \modM \rightarrow \modN $ that is
also a strict braided monoidal functor.

\subsection{Strict symmetric pivotal categories}

For a strict symmetric pivotal category $\modM $, we have
\begin{gather*}
  c^L_X=(c^R_X)^{-1}
\end{gather*}
for any object $X$ in $\modM $.  Instead of $c^L$ and $c^R$, we use the
following notation\\

\begin{gather}
  \label{e22}
\gamma _X(=c^R_X)= \quad {\raisebox{-7.mm}{\begin{overpic}[width=10mm]{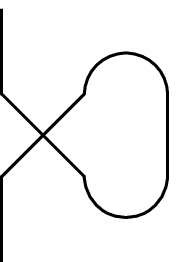}\put(-4,45){$X$}\put(-4,-7){$X$}\end{overpic}}}\quad,\qquad
\gamma _X^{-1}(=c^L_X)= \quad {\raisebox{-7.mm}{\begin{overpic}[width=10mm, angle=180]{rightcurl.eps}\put(25,45){$X$}\put(25,-9){$X$}\end{overpic}}}\quad.
\end{gather}
\vskip1em
We call $\gamma _X$ (resp. $\gamma _X^{-1}$) the {\em right} (resp. {\em left})
{\em curl} for $X$.
The $\gamma _X\zzzcolon X\rightarrow X$ forms a monoidal natural isomorphism, i.e., we have
$\gamma _I=\id_I$ and
\begin{gather*}
  \gamma _{X\otimes Y}=\gamma _X\otimes \gamma _Y
\end{gather*}
for any object $X$ and $Y$ in $\modM $.
We also have
\begin{gather*}
  \gamma _{X^*}= (\gamma _X^{-1})^*.
\end{gather*}

\section{Pivotal objects and ribbon objects}

In this section, we introduce the notion of pivotal objects in a monoidal category.
Pivotal objects in a monoidal category $\modC$ form a pivotal category even if the $\modC$ is not rigid.
Then we define a ribbon object in a braided monoidal category, which is a pivotal object such that the associated left and right positive curls coincide.
Ribbon objects in a braided category form a ribbon category.

\subsection{Pivotal objects}

Let $\modM $ be a strict monoidal category.  A {\em pivotal object} in
$\modM $ is a sextuple $\XX{X}$ consisting of objects $X$ and $X^*$ and
morphisms
\begin{gather*}
\ev_X= \ {\raisebox{-2.mm}{\begin{overpic}[width=10mm]{bottom.eps}\put(-4,18){$X^*$}\put(23,18){$X$}\end{overpic}}} \ \ ,\quad
\coev_X= \ {\raisebox{.mm}{\begin{overpic}[width=10mm, angle=180]{bottom.eps}\put(-4,-10){$X$}\put(23,-10){$X^*$}\end{overpic}}} \ \ \ , \quad
\ev_{X^*}= \ {\raisebox{-2.mm}{\begin{overpic}[width=10mm]{bottom.eps}\put(-4,18){$X$}\put(23,18){$X^*$}\end{overpic}}} \ \ ,\quad
\coev_{X^*}= \ {\raisebox{.mm}{\begin{overpic}[width=10mm, angle=180]{bottom.eps}\put(-4,-10){$X^*$}\put(23,-10){$X$}\end{overpic}}}
\end{gather*}
such that both $(X,X^*,\ev_{X},\coev_{X})$ and
$(X^*,X,\ev_{X^*},\coev_{X^*})$ are dualities in $\modM $.  For simplicity
of notation, we denote it by
\begin{gather*}
  X^p=\XX{X}.
\end{gather*}

Let $X^p$ and $Y^p$ be pivotal objects in $\modM $ and let $f\zzzcolon X\rightarrow Y$ be a
morphism in $\modM $.  The left dual $f^*\zzzcolon Y^*\rightarrow X^*$ of $f$ is defined by
\eqref{e10}.  By using the dualities $(X^*,X,\ev_{X^*},\coev_{X^*})$
and $(Y^*,Y,\ev_{Y^*},\coev_{Y^*})$, we can further obtain \\
\begin{gather} 
  \label{e30}
  f^{**}=(f^*)^*= \ {\raisebox{-10.mm}{\begin{overpic}[width=20mm]{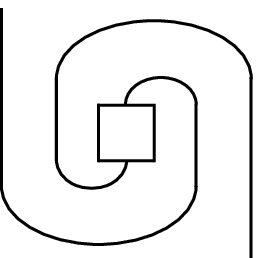}\put(53,-8){$Y$}\put(-3,60){$X$}\put(25,27){$f$}\end{overpic}}} \ \ . 
\end{gather} \\
A morphism $f\zzzcolon X^p\rightarrow Y^p$ between pivotal objects $X^p$ and $Y^p$ is
defined to be a morphism $f\zzzcolon X\rightarrow Y$ satisfying $f=f^{**}$.
It is easy to check that pivotal objects in $\modM $ and morphisms between
them form a category, which is denoted by $\modM ^p$.

\begin{remark}
  \label{r22}
  Our definition of pivotal objects is similar to the notion of
  ``pivotal objects'' in a rigid monoidal category defined by Shimizu
  \cite{Shimizu}.  If $\modM $ is rigid, then pivotal objects and the
  category $\modM ^p$ may be considered as a strict version of ``pivotal
  objects'' and the ``pivotal cover'' of $\modM $ defined in
  \cite{Shimizu}.
\end{remark}

\begin{proposition}
The category $\modM ^p$ has a structure of a strict pivotal category as
follows.  The unit object in $\modM ^p$ is
$I^p:=(I,I,\id_I,\id_I,\id_I,\id_I)$.  The tensor product of
$X^p,Y^p\in \modM ^p$ is given by
\begin{gather*}
  X^p\otimes Y^p = (X\otimes Y,Y^*\otimes X^*,\ev_{X\otimes Y},\coev_{X\otimes Y},\ev_{Y^*\otimes X^*},\coev_{Y^*\otimes X^*}),
\end{gather*}
where
\begin{gather*}
\ev_{X\otimes Y}= \ {\raisebox{-2.mm}{\begin{overpic}[width=10mm]{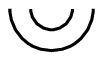}\put(-6,18){\tiny $Y^* X^*$}\put(19,18){\tiny $X$ $Y$}\end{overpic}}} \ \ ,\quad
\coev_{X\otimes Y}= \ {\raisebox{.mm}{\begin{overpic}[width=10mm, angle=180]{doublebottom.eps}\put(-4,-7){\tiny $X$ $Y$}\put(18,-7){\tiny $Y^* X^*$}\end{overpic}}} \hspace{.4cm} , \quad
\ev_{Y^*\otimes X^*}= \ {\raisebox{-2.mm}{\begin{overpic}[width=10mm]{doublebottom.eps}\put(-4,18){\tiny $X$ $Y$}\put(17,18){\tiny $Y^* X^*$}\end{overpic}}} \  \ ,\quad
\coev_{Y^*\otimes X^*}= \ {\raisebox{.mm}{\begin{overpic}[width=10mm, angle=180]{doublebottom.eps}\put(-5,-7){\tiny $Y^* X^*$}\put(18,-7){\tiny $X$ $Y$}\end{overpic}}}\ \ .
\end{gather*}
The tensor product of morphisms in $\modM ^p$ is the tensor product in
$\modM $.  The left dual $(X^p)^*$ of $X^p\in \modM ^p$ is given by
\begin{gather*}
  (X^p)^*=(X^*,X,\ev_{X^*},\coev_{X^*},\ev_X,\coev_X).
\end{gather*}
The evaluation and coevaluation morphisms for $X^p\in \modM ^p$ are given by
\begin{gather*}
  \ev_{X^p}=\ev_X,\quad
  \coev_{X^p}=\coev_X,\quad
  \ev_{(X^p)^*}=\ev_{X^*},\quad
  \coev_{(X^p)^*}=\coev_{X^*}.\quad
\end{gather*}
It follows that the dual $f^*\zzzcolon (Y^p)^*\rightarrow (X^p)^*$ of any morphism
$f\zzzcolon X^p\rightarrow Y^p$ in $\modM ^p$ is the same as the dual $f^*\zzzcolon Y^*\rightarrow X^*$ of
$f\zzzcolon X\rightarrow Y$ in $\modM $.  
\end{proposition}

\begin{proof}
It is straightforward to check that $\modM ^p$ is a
strict pivotal category. 
\end{proof}

We call $\modM ^p$ the {\em pivotalization} of
$\modM $.

There is an obvious ``forgetful'' functor $U\zzzcolon \modM ^p\rightarrow \modM $ which sends each
$X^p\in \modM ^p$ to $X\in \modM $ and each morphism $f$ to itself.

\subsection{Ribbon objects in braided categories}

\begin{proposition}
Let $\modM$ be a strict braided monoidal category.  
Then $\modM ^p$ is a strict braided pivotal category 
and the forgetful functor $U\zzzcolon \modM ^p\rightarrow \modM $ is a strict
braided monoidal functor.
\end{proposition}

\begin{proof}
A braiding $\psi _{X^p,Y^p}\zzzcolon X^p \otimes Y^p \rightarrow Y^p \otimes X^p$ between pivotal objects $X^p$ and $Y^p$ is defined to be a morphism $\psi _{X,Y} \zzzcolon X \otimes Y \rightarrow Y \otimes X$. Then, $\modM ^p$ is a strict braided pivotal category.
Moreover, the functor $U\zzzcolon \modM ^p\rightarrow \modM $ is naturally a strict
braided monoidal functor.
\end{proof}

A pivotal object $X^p$ in $\modM $ is called a {\em ribbon object} in $\modM $
if the following ribbon condition is satisfied.
\begin{gather}
  \label{e-ribbon}
    {\raisebox{-7.mm}{\begin{overpic}[width=10mm, angle=180]{rightpositivecurl.eps}\put(24,45){$X$}\put(24,-9){$X$}\end{overpic}}} \quad
= \quad  {\raisebox{-7.mm}{\begin{overpic}[width=10mm]{rightpositivecurl.eps}\put(-4,45){$X$}\put(-4,-9){$X$}\end{overpic}}} \quad \zzzcolon X\rightarrow X.
\end{gather} \\

Let $\modM ^r$ denote the full subcategory of $\modM ^p$ whose objects
are the ribbon objects in $\modM $.

\begin{proposition}
  \label{r24}
  Let $\modM $ be a strict braided monoidal category.  Then $\modM ^r$ is a 
  strict braided pivotal subcategory of $\modM ^p$ and a strict ribbon
  category.
\end{proposition}

\begin{proof}
  To show that $\modM ^r$ is a strict braided pivotal subcategory of
  $\modM ^p$, it suffices to check the following.
\begin{enumerate}
\item The tensor product of two ribbon objects in $\modM ^p$ is a ribbon
  object in $\modM ^p$, because \\
\begin{gather*}
      {\raisebox{-7.mm}{\begin{overpic}[height=20mm, angle=0]{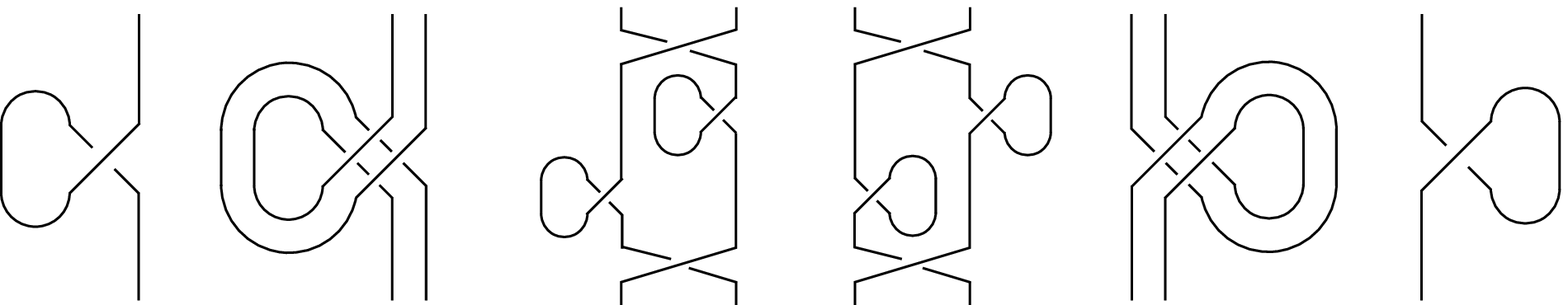}\put(10,60){$X \otimes   Y$}\put(10,-9){$X \otimes Y$}\put(70,60){$X$}\put(80,60){$Y$}\put(70,-9){$X$}\put(80,-9){$Y$}\put(255,60){$X \otimes Y$}\put(255,-9){$X \otimes Y$}\put(30,25){$=$}\put(89,25){$=$}\put(146,25){$=$}\put(202,25){$=$}\put(257,25){$=$}\end{overpic}}} \quad .
  \end{gather*} \\

\item The unit object $I^p$ of $\modM ^p$ is a ribbon object in $\modM ^p$, because \\ 
$c^R_I=(I \otimes d_{I^*})(\psi_{I,I}\otimes I^*)(I \otimes b_I)=(I \otimes d_{I})(\psi_{I,I}\otimes I)(I \otimes b_I)=\id_I$ \\and similarly $c^L_I=\id_I$. 
\item The dual of any ribbon object in $\modM ^p$ is a ribbon object, from the left duals of \eqref{e11}.
\end{enumerate}
Now $\modM ^r$ is a pivotal category.  By the definition of ribbon
objects, each object in $\modM ^r$ satisfies the ribbon condition
\eqref{e11}.  Hence $\modM ^r$ is a ribbon category.
\end{proof}

\section{Hopf algebras and Yetter--Drinfeld modules}
In this section, we first recall notions of Hopf algebras and Yetter--Drinfeld modules in a symmetric monoidal category, and define ribbon Yetter--Drinfeld modules to be ribbon objects in the category of Yetter--Drinfeld modules.

\subsection{Hopf algebras in a symmetric monoidal category}
\label{sec:hopf-algebr-symm}
Let $\modV $ be a strict symmetric monoidal category.

A {\em Hopf algebra} $H=(H,\mu ,\eta ,\Delta ,\epsilon ,S)$ in $\modV $ consists of an
object $H$ in $\modV $ and morphisms
\begin{gather*}
\mu= \ {\raisebox{-5mm}{\begin{overpic}[width=8mm, angle=180]{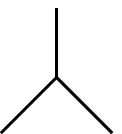}\put(-4,28){$H$}\put(19,28){$H$}\put(6,-9){$H$}\end{overpic}}} \quad \zzzcolon H\otimes H\rightarrow H,\quad
  \eta= \ {\raisebox{-6mm}{\begin{overpic}[width=.7mm, angle=180]{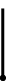}\put(-3,-9){$H$}\end{overpic}}} \quad \zzzcolon I\rightarrow H,\quad \\
\newline \\
\newline \\
  \Delta= \ {\raisebox{-5mm}{\begin{overpic}[width=8mm]{Delta.eps}\put(-5,-9){$H$}\put(18,-9){$H$}\put(8,28){$H$}\end{overpic}}} \quad \zzzcolon H\rightarrow H\otimes H,\quad
  \epsilon= \ {\raisebox{-2mm}{\begin{overpic}[width=.7mm]{epsilon.eps}\put(-4,24){$H$}\end{overpic}}} \quad \zzzcolon H\rightarrow I,\quad 
  S= \ {\raisebox{-4mm}{\begin{overpic}[width=5mm, angle=180]{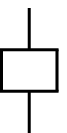}\put(3,33){$H$}\put(3,-9){$H$}\put(4,12){$S$}\end{overpic}}} \quad \zzzcolon H\rightarrow H
\end{gather*}
such that \\
\begin{gather}
  \label{e12}
{\raisebox{-5mm}{\begin{overpic}[width=10mm, angle=180]{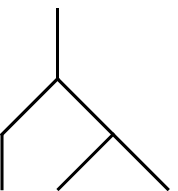}\put(-3,33){$H$}\put(15,33){$H$}\put(25,33){$H$}\put(13,-9){$H$}\put(40,10){$=$}\end{overpic}}}  \hspace{1.3cm}
{\raisebox{-5mm}{\begin{overpic}[width=10mm, angle=180]{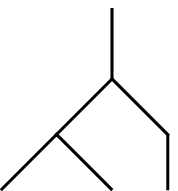}\put(-3,33){$H$}\put(8,33){$H$}\put(25,33){$H$}\put(5,-9){$H$}\end{overpic}}} \quad,  \hspace{1.cm}
{\raisebox{-5mm}{\begin{overpic}[width=5mm, angle=180]{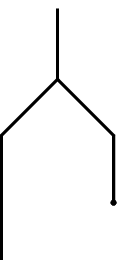}\put(10,33){$H$}\put(3,-9){$H$}\put(-7,23){$\eta$}\put(28,10){$=$}\end{overpic}}}  \hspace{1.3cm}
{\raisebox{-5mm}{\begin{overpic}[width=.3mm, angle=180]{isotopy1.eps}\put(-3,33){$H$}\put(-3,-9){$H$}\put(17,10){$=$}\end{overpic}}}  \hspace{1.3cm}
{\raisebox{-5mm}{\begin{overpic}[width=5mm, angle=180]{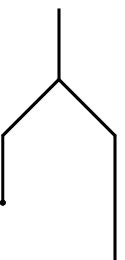}\put(-3,33){$H$}\put(3,-9){$H$}\put(16,23){$\eta$}\end{overpic}}} \quad ,   \\
\newline \nonumber \\
\newline \nonumber \\
{\raisebox{-5mm}{\begin{overpic}[width=10mm, angle=0]{coalgebrarelation1.eps}\put(-3,-9){$H$}\put(15,-9){$H$}\put(25,-9){$H$}\put(15,33){$H$}\put(40,10){$=$}\end{overpic}}}  \hspace{1.3cm}
{\raisebox{-5mm}{\begin{overpic}[width=10mm, angle=0]{coalgebrarelation1rev.eps}\put(-3,-9){$H$}\put(8,-9){$H$}\put(25,-9){$H$}\put(5,33){$H$}\end{overpic}}},  \hspace{1.cm}
{\raisebox{-5mm}{\begin{overpic}[width=5mm, angle=0]{coalgebrarelation2.eps}\put(12,-9){$H$}\put(3,33){$H$}\put(28,10){$=$}\put(-7,5){$\epsilon$}\end{overpic}}}  \hspace{1.3cm}
{\raisebox{-5mm}{\begin{overpic}[width=.3mm, angle=180]{isotopy1.eps}\put(-3,-9){$H$}\put(-3,33){$H$}\put(17,10){$=$}\end{overpic}}}  \hspace{1.3cm}
{\raisebox{-5mm}{\begin{overpic}[width=5mm, angle=0]{coalgebrarelation2rev.eps}\put(-3,-9){$H$}\put(3,33){$H$}\put(16,5){$\epsilon$}\end{overpic}}} \quad ,   \\
\newline \nonumber \\
\newline \nonumber \\
    \label{e14}
{\raisebox{-2.5mm}{\begin{overpic}[width=.5mm, angle=0]{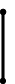}\put(4,18){$\eta$}\put(4,-5){$\epsilon$}\end{overpic}}} \hspace{.5cm} = (\text{empty}),\quad
{\raisebox{0.mm}{\begin{overpic}[width=7.mm, angle=0]{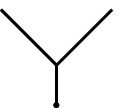}\put(13,-1){$\epsilon$}\put(18,20){$H$}\put(-4,20){$H$}\end{overpic}}}   \hspace{.5cm} = \hspace{.5cm}
{\raisebox{.5mm}{\begin{overpic}[width=.5mm, angle=0]{epsilon.eps}\put(4,-3){$\epsilon$}\put(-4,18){$H$}\end{overpic}}} \hspace{.5cm}
{\raisebox{.5mm}{\begin{overpic}[width=.5mm, angle=0]{epsilon.eps}\put(4,-3){$\epsilon$}\put(-4,18){$H$}\end{overpic}}} \quad , \quad
{\raisebox{-5.mm}{\begin{overpic}[width=7.mm, angle=180]{Hopfalgebrarelation1.eps}\put(13,15){$\eta$}\put(18,-9){$H$}\put(-4,-9){$H$}\end{overpic}}}   \hspace{.5cm} = \hspace{.5cm}
{\raisebox{-4.mm}{\begin{overpic}[width=.5mm, angle=180]{epsilon.eps}\put(3,15){$\eta$}\put(-4,-9){$H$}\end{overpic}}} \hspace{.5cm}
{\raisebox{-4.mm}{\begin{overpic}[width=.5mm, angle=180]{epsilon.eps}\put(3,15){$\eta$}\put(-4,-9){$H$}\end{overpic}}} \quad , \\
\newline \nonumber \\
\newline \nonumber \\
  \label{e15}
{\raisebox{-6.mm}{\begin{overpic}[width=30.mm, angle=0]{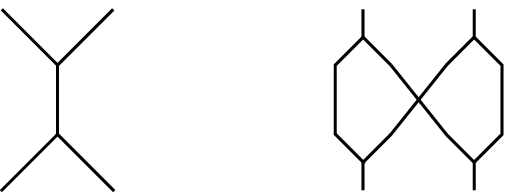}\put(-4,35){$H$}\put(18,35){$H$}\put(18,-9){$H$}\put(-4,-9){$H$}\put(57,35){$H$}\put(78,35){$H$}\put(57,-9){$H$}\put(78,-9){$H$}\put(40,16){$=$}\end{overpic}}} \quad, \\  
  \label{e16}
\newline \nonumber \\
\newline \nonumber \\
{\raisebox{-10.mm}{\begin{overpic}[width=10.mm, angle=0]{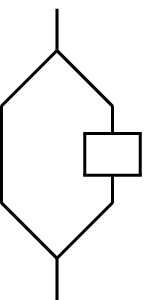}\put(6,-9){$H$}\put(7,62){$H$}\put(19,26.5){$S$}\end{overpic}}} \hspace{.5cm} = \hspace{.5cm}
{\raisebox{-10.mm}{\begin{overpic}[width=10.mm, angle=180]{antipoderelation.eps}\put(14,-9){$H$}\put(14,62){$H$}\put(3,26){$S$}\end{overpic}}} \hspace{.5cm} = \hspace{1cm}
{\raisebox{-10.mm}{\begin{overpic}[width=.45mm, angle=0]{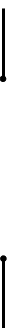}\put(-3,-9){$H$}\put(-3,62){$H$}\put(3,13){$\eta$}\put(3,43){$\epsilon$}\end{overpic}}} \hspace{.5cm} \quad . 
\end{gather} \\
\begin{remark}
It is known that the following statements are from above relations. \\
\begin{gather}
  \label{e16-2}
{\raisebox{-7.mm}{\begin{overpic}[width=13.mm, angle=0]{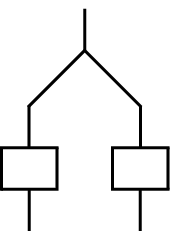}\put(2,-9){$H$}\put(26,-9){$H$}\put(14,52){$H$}\put(28,11){$S$}\put(2,11){$S$}\end{overpic}}} \hspace{.5cm} = \hspace{.5cm}
{\raisebox{-7.mm}{\begin{overpic}[width=9.mm, angle=0]{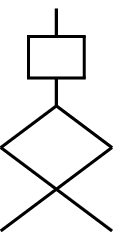}\put(-5,-9){$H$}\put(18,-9){$H$}\put(8,52){$H$}\put(9,37){$S$} \end{overpic}}} \hspace{.5cm} \ , \hspace{.5cm}
{\raisebox{0.mm}{\begin{overpic}[width=5.mm, angle=0]{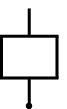}\put(3,28){$H$}\put(4,10){$S$}\put(11,0){$\epsilon$}\end{overpic}}} \hspace{.5cm} = \hspace{.5cm}
{\raisebox{0.mm}{\begin{overpic}[width=.8mm, angle=0]{epsilon.eps}\put(-2,28){$H$}\put(6,0){$\epsilon$}\end{overpic}}} \hspace{.2cm} \quad , \hspace{.2cm} \\
\newline \nonumber \\
{\raisebox{-7.mm}{\begin{overpic}[width=13.mm, angle=180]{Delta-antipode.eps}\put(2,50){$H$}\put(26,50){$H$}\put(14,-9){$H$}\put(28,32){$S$}\put(2,32){$S$}\end{overpic}}} \hspace{.5cm} = \hspace{.5cm}
{\raisebox{-7.mm}{\begin{overpic}[width=9.mm, angle=180]{antipode-Delta.eps}\put(-5,52){$H$}\put(18,52){$H$}\put(8,-9){$H$}\put(9,7){$S$} \end{overpic}}} \hspace{.5cm} \ , \hspace{.5cm}
{\raisebox{-7.mm}{\begin{overpic}[width=5.mm, angle=180]{antipode-epsilon.eps}\put(3,-11){$H$}\put(4,8.5){$S$}\put(11,22){$\eta$}\end{overpic}}} \hspace{.5cm} = \hspace{.5cm}
{\raisebox{-7.mm}{\begin{overpic}[width=.8mm, angle=180]{epsilon.eps}\put(-2,-9){$H$}\put(4,22){$\eta$}\end{overpic}}} \hspace{.5cm}.
\end{gather} \\
\end{remark}
In what follows, we always assume that the antipode $S$ is invertible.
This means that there is a morphism $S^{-1}\zzzcolon H\rightarrow H$ satisfying \\
\begin{gather}
  \label{e17}
{\raisebox{-7.mm}{\begin{overpic}[width=8.mm, angle=180]{composition.eps}\put(6,-9){$H$}\put(8,50){$H$}\put(3,31){$S^{-1}$}\put(7,10){$S$}\end{overpic}}} \hspace{.5cm} = \hspace{.5cm}
{\raisebox{-7.mm}{\begin{overpic}[width=8.mm, angle=180]{composition.eps}\put(6,-9){$H$}\put(3,8.5){$S^{-1}$}\put(8,50){$H$}\put(7,31){$S$}\end{overpic}}} \hspace{.5cm} = \hspace{.5cm}
{\raisebox{-6mm}{\begin{overpic}[width=.45mm, angle=180]{isotopy1.eps}\put(-3,50){$H$}\put(-3,-9){$H$}\end{overpic}}} \quad . 
\end{gather} \\

The multi-input multiplications and comultiplications \\
\begin{gather*}
\mu _n = \quad {\raisebox{-5mm}{\begin{overpic}[width=30mm, angle=0]{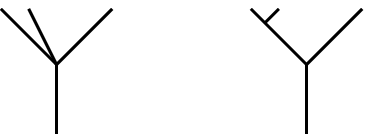}\put(-5,32){$H$}\put(5,32){$H$}\put(23,32){$H$}\put(8,-9){$H$}\put(37,14){$:=$}\put(55,32){$H$}\put(65,32){$H$}\put(80,32){$H$}\put(10,28){$\dots$}\put(68,28){$\dots$}\put(66,-9){$H$}\end{overpic}}} \quad \zzzcolon H^{\otimes n} \rightarrow H, \\
\newline \\
\newline \\
\Delta _n =  \quad {\raisebox{-5mm}{\begin{overpic}[width=30mm, angle=0]{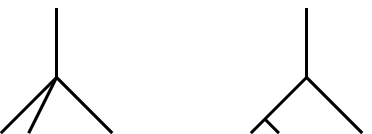}\put(-5,-9){$H$}\put(5,-9){$H$}\put(23,-9){$H$}\put(9,32){$H$}\put(10,4){$\dots$}\put(68,4){$\dots$}\put(37,14){$:=$}\put(55,-9){$H$}\put(65,-9){$H$}\put(80,-9){$H$}\put(68,32){$H$}\end{overpic}}} \quad \zzzcolon H\rightarrow H^{\otimes n} \\
\end{gather*}
for $n\ge 0$ are defined inductively by
\begin{gather*}
  \mu _0=\eta ,\quad \mu _1=\id_H,\quad \mu _n=\mu (\mu _{n-1}\otimes H)\quad
  (n\ge 2),\\
  \Delta _0=\epsilon ,\quad \Delta _1=\id_H,\quad
  \Delta _n=(\Delta _{n-1}\otimes H)\Delta \quad (n\ge 2).
\end{gather*}
These morphisms satisfy the generalized (co)associativity relations: \\
\begin{gather}
{\raisebox{-5mm}{\begin{overpic}[width=40mm, angle=0]{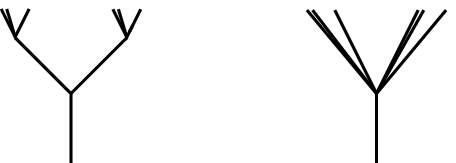}\put(-5,45){$H^{\otimes k_1}$}\put(28,45){$H^{\otimes k_m}$}\put(15,-8){$H$}\put(12,38){$\dots$}\put(86,49){$\dots$}\put(50,14){$=$}\put(111,45){$H$}\put(88,40){$\dots$}\put(70,45){$H$}\put(90,-8){$H$}\end{overpic}}} \quad, \hspace{1.cm}
{\raisebox{-5mm}{\begin{overpic}[width=40mm, angle=0]{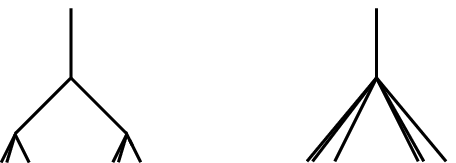}\put(-5,-9){$H^{\otimes k_1}$}\put(28,-9){$H^{\otimes k_m}$}\put(15,45){$H$}\put(12,5){$\dots$}\put(90,5){$\dots$}\put(50,14){$=$}\put(68,-9){$H$}\put(111,-9){$H$}\put(90,-6){$\dots$}\put(92,45){$H$}\end{overpic}}} \quad, \hspace{1.cm} 
\end{gather} \\
for $m\ge 0,k_1,\ldots,k_m\ge 0$.

\subsection{Yetter--Drinfeld modules}
\label{sec:yett-drinf-modul}

Let $H$ be a Hopf algebra in a strict symmetric monoidal category $\modV $.

A (left-left) {\em Yetter--Drinfeld module} over $H$ in $\modV $ is a
triple $(X,\alpha ,\beta )$ of an object $X$ in $\modV $ and morphisms \\ 
\begin{gather*}
{\alpha = \quad \raisebox{-6.mm}{\begin{overpic}[height=12mm]{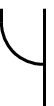}\put(-4,36){$H$}\put(14,36){$X$}\end{overpic}}} \quad \zzzcolon H\otimes X\rightarrow X,\quad\quad \beta = \quad \raisebox{-6.mm}{\begin{overpic}[height=12mm]{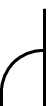}\put(-4,-8){$H$}\put(14,-8){$X$}\end{overpic}}\quad
\zzzcolon X\rightarrow H\otimes X
\end{gather*}
satisfying the following conditions.
\begin{enumerate}
\item $(X,\alpha )$ is a left $H$-module, i.e., we have \\
\begin{gather*}
\raisebox{-6.mm}{\begin{overpic}[height=12mm]{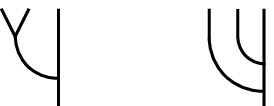}\put(-4,36){$H$}\put(6,36){$H$}\put(15,36){$X$}\put(15,-8){$X$}\put(40,17){$=$}\put(63,36){$H$}\put(73,36){$H$}\put(83,36){$X$}\put(83,-8){$X$}\put(23,6){$\alpha$}\put(90,5){$\alpha$}\put(90,13){$\alpha$}\end{overpic}}\quad\quad , \hspace{1.5cm} \raisebox{-6.mm}{\begin{overpic}[height=12mm]{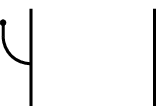}\put(8,36){$X$}\put(8,-8){$X$}\put(50,36){$X$}\put(50,-8){$X$}\put(30,17){$=$}\put(14,11){$\alpha$}\end{overpic}}\quad .
\end{gather*} \\
\item $(X,\beta )$ is a left $H$-comodule, i.e., we have \\
\begin{gather*}
\raisebox{-6.mm}{\begin{overpic}[height=12mm]{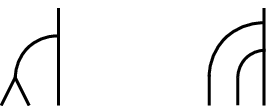}\put(-4,-8){$H$}\put(6,-8){$H$}\put(15,36){$X$}\put(15,-8){$X$}\put(40,17){$=$}\put(63,-8){$H$}\put(73,-8){$H$}\put(83,36){$X$}\put(83,-8){$X$}\put(23,20){$\beta$}\put(90,25){$\beta$}\put(90,15){$\beta$}\end{overpic}}\quad\quad , \hspace{1.5cm} \raisebox{-6.mm}{\begin{overpic}[height=12mm]{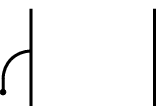}\put(8,36){$X$}\put(7,-8){$X$}\put(50,36){$X$}\put(49,-8){$X$}\put(30,17){$=$}\put(14,13){$\beta$}\end{overpic}}\quad .
\end{gather*} \\
\item We have 
\begin{gather*}
 \raisebox{-7.mm}{\begin{overpic}[height=17mm]{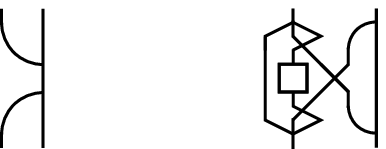}\put(-4,-8){$H$}\put(-4,50){$H$}\put(12,50){$X$}\put(12,-8){$X$}\put(50,23){$=$}\put(94,22){$S$}\put(92,-8){$H$}\put(92,50){$H$}\put(120,-8){$X$}\put(120,50){$X$}\put(18,29){$\alpha$}\put(18,15){$\beta$}\put(128,5){$\alpha$}\put(128,40){$\beta$}\end{overpic}}\quad\quad .
\end{gather*} \\
\end{enumerate}

A {\em morphism} of Yetter--Drinfeld modules
\begin{gather*}
  f\zzzcolon (X,\alpha ,\beta )\rightarrow (X',\alpha ',\beta ')
\end{gather*}
is a morphism $f\zzzcolon X\rightarrow X'$ in $\modV $ that is both a left $H$-module morphism
$f\zzzcolon (X,\alpha )\rightarrow (X',\alpha ')$ and a left $H$-comodule morphism
$f\zzzcolon (X,\beta )\rightarrow (X',\beta ')$, i.e., we have \\
\begin{gather*}
 \raisebox{-5.mm}{\begin{overpic}[height=12mm]{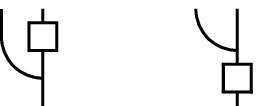}\put(12,22){$f$}\put(-4,36){$H$}\put(12,36){$X$}\put(12,-8){$X'$}\put(45,15){$=$}\put(79,-8){$X'$}\put(64,36){$H$}\put(80,36){$X$}\put(80,7){$f$}\put(19,8){$\alpha'$}\put(87,19){$\alpha$}\end{overpic}}\quad\quad , \hspace{1.5cm} \raisebox{-5.mm}{\begin{overpic}[height=12mm]{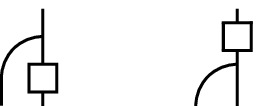}\put(12,7){$f$}\put(-4,-8){$H$}\put(12,36){$X$}\put(12,-8){$X'$}\put(40,15){$=$}\put(80,-8){$X'$}\put(64,-8){$H$}\put(80,36){$X$}\put(80,22){$f$}\put(18,21){$\beta$}\put(86,9){$\beta'$}\end{overpic}}\quad\quad . 
\end{gather*} \\

Let $\YDHV$ denote the category of Yetter--Drinfeld modules over $H$ in
$\modV $ and morphisms of Yetter--Drinfeld modules.  The category $\YDHV$ has a
structure of a strict braided monoidal category as follows
(\cite{Yetter,Drinfeld1}).  The monoidal unit of $\YDHV$ is
$(I,\epsilon ,\eta )$.  The tensor product of $(X,\alpha ,\beta )$ and $(X',\alpha ',\beta ')$ is
given by
\begin{gather*}
  (X,\alpha ,\beta )\otimes (X',\alpha ',\beta ')=(X\otimes X',\alpha '',\beta ''),
\end{gather*}
where
\begin{gather}
  \label{e2}
\alpha '' = \quad \raisebox{-5.mm}{\begin{overpic}[height=12mm]{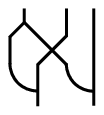}\put(2,36){$H$}\put(14,36){$X$}\put(7,-8){$X$}\put(24,-8){$X'$}\put(25,36){$X'$}\put(-5,5){$\alpha$}\put(30,5){$\alpha'$}\end{overpic}}\quad \quad, \quad\quad
\beta '' = \quad \raisebox{-5.mm}{\begin{overpic}[height=12mm]{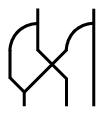}\put(2,-8){$H$}\put(14,-8){$X$}\put(7,36){$X$}\put(24,36){$X'$}\put(25,-8){$X'$}\put(-0,28){$\beta$}\put(30,28){$\beta'$}\end{overpic}}\quad \quad .
\end{gather}
\medskip
The braiding
\begin{gather*}
  \psi _{(X,\alpha ,\beta ),(X',\alpha ',\beta ')}\zzzcolon (X,\alpha ,\beta )\otimes (X',\alpha ',\beta ')\rightarrow (X',\alpha ',\beta ')\otimes (X,\alpha ,\beta )
\end{gather*}
and its inverse
\begin{gather*}
  \psi _{(X,\alpha ,\beta ),(X',\alpha ',\beta ')}^{-1}\zzzcolon (X',\alpha ',\beta ')\otimes (X,\alpha ,\beta )\rightarrow (X,\alpha ,\beta )\otimes (X',\alpha ',\beta ')
\end{gather*}
are given by
\begin{gather*}
\psi _{(X,\alpha ,\beta ),(X',\alpha ',\beta ')}= \quad \raisebox{-7.mm}{\begin{overpic}[height=15mm]{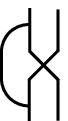}\put(6,45){$X$}\put(19,45){$X'$}\put(6,-8){$X'$}\put(19,-8){$X$}\put(-5,3){$\alpha'$}\put(-5,38){$\beta$}\end{overpic}}\quad \quad,
\quad\quad
\psi _{(X,\alpha ,\beta ),(X',\alpha ',\beta ')}^{-1}= \quad \raisebox{-9.5mm}{\begin{overpic}[height=20mm]{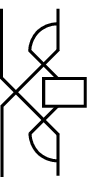}\put(-6,59){$X'$}\put(15,59){$X$}\put(-6,-8){$X$}\put(15,-8){$X'$}\put(14,27){\tiny $S^{-1}$}\put(22,5){$\alpha'$}\put(22,50){$\beta$}\end{overpic}}\quad \quad.
\end{gather*} \\

\subsection{Ribbon Yetter--Drinfeld modules}

A {\em ribbon Yetter--Drinfeld module} is a ribbon object in the strict braided
monoidal category $\YDHV$ of Yetter--Drinfeld modules over $H$ in $\modV $.  The {\em
category of ribbon Yetter--Drinfeld modules} is $\rYDHV:=(\YDHV)^r$, the category of
ribbon objects in $\YDHV$.
Proposition \ref{r24} immediately implies the following.
\begin{theorem}
  \label{r34}
  The category $\rYDHV$ is a strict ribbon category.
\end{theorem}

We have a simpler definition of ribbon Yetter--Drinfeld modules as follows.

\begin{proposition}
  \label{r35}
  A ribbon Yetter--Drinfeld module is equivalent to a pair of
  \begin{itemize}
  \item a Yetter--Drinfeld module $(X,\alpha _X,\beta _X)$ over $H$ in $\modV $, and
  \item a pivotal object $(X,X^*,\ev_X,\coev_X,\ev_{X^*},\coev_{X^*})$
    for $X$ in $\modV $
  \end{itemize}
  that satisfies 
  \begin{gather}
    \label{e19} \gamma _X\alpha _X=\alpha _X(S^2\otimes \gamma _X),\\
    \label{e20}\beta _X\gamma _X=(S^2\otimes \gamma _X)\beta _X,\\
    \label{e21} c_X^R=c_X^L,
  \end{gather}
  where $\gamma _X\zzzcolon X\rightarrow X$ is the right curl for $X$ defined by \eqref{e22},
  and $c_X^R,c_X^L\zzzcolon X\rightarrow X$ are the right and left positive curls defined by
  \eqref{e23}.
\end{proposition}

By this proposition, we may regard
$(X,X^*,\ev_X,\coev_X,\ev_{X^*},\coev_{X^*},\alpha _X,\beta _X)$ satisfying the
above conditions as a ribbon Yetter--Drinfeld module.

\begin{proof}
  Let $Z$ denote the set of the data
  $(X,X^*,\ev_X,\coev_X,\ev_{X^*},\coev_{X^*},\alpha _X,\beta _X)$ satisfying
  the conditions given in the proposition.  We will construct a
  bijection between $Z$ and $\Ob(\rYDHV)$.
  
  If $\hat X=((X,\alpha _X,\beta _X),(X^*,\alpha _{X^*},\beta _{X^*}),\evv
  X)\in \Ob(\rYDHV)$, then clearly $(X,\alpha _X,\beta _X)$ is a Yetter--Drinfeld module and
  $(X,X^*,\ev_X,\coev_X,\ev_{X^*},\coev_{X^*})$ is a pivotal object in
  $\modV $.  
  Since $\ev_{X^*}$ and $\coev_{X}$ are morphisms of Yetter--Drinfeld modules,
  we have \\
  \begin{gather*}
 {\raisebox{-14mm}{\begin{overpic}[height=30mm, angle=0]{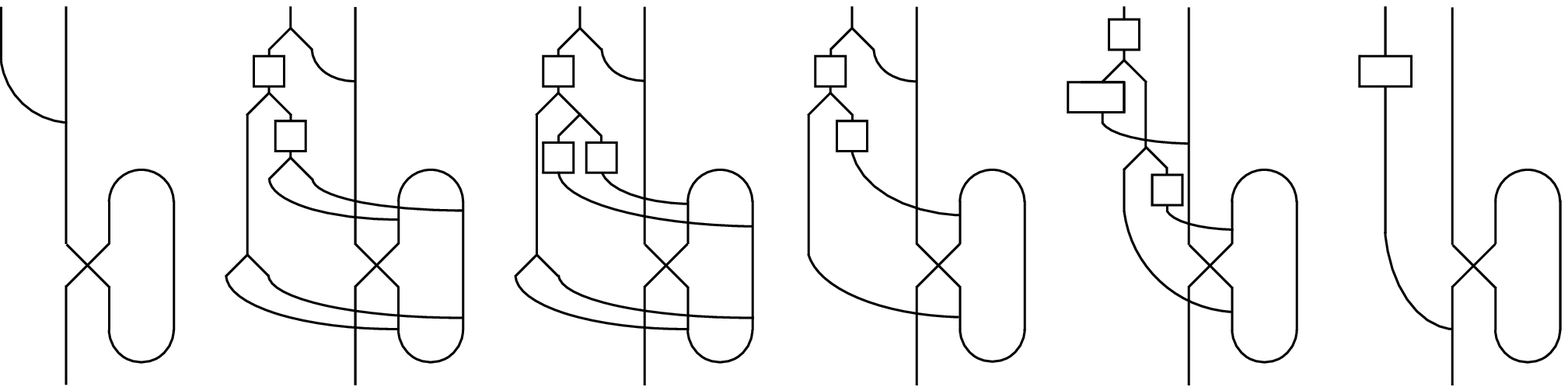}\put(-4,87){$H$}\put(11,87){$X$}\put(11,-9){$X$}\put(18,58){$\alpha_X$}\put(43,40){$=$}\put(57,69){\tiny $S$}\put(62,54){\tiny $S$}\put(107,40){$=$}\put(122,68){\tiny $S$}\put(122,49){\tiny $S$}\put(132,49){\tiny $S$}\put(170,40){$=$}\put(182,69){\tiny $S$}\put(187,54){\tiny $S$}\put(233,40){$=$}\put(248,77){\tiny $S$}\put(237,62){\tiny $S^{-1}$}\put(258,42){\tiny $S$}\put(293,40){$=$}\put(304,68){\tiny $S^2$}\put(304,87){$H$}\put(319,87){$X$}\put(319,-9){$X$}\end{overpic}}} \quad.
 \end{gather*} \\
 Hence we have \eqref{e19}.
  Similarly, \eqref{e20} follows since $\ev_{X}$ and $\coev_{X^*}$ are morphisms of
  Yetter--Drinfeld modules.  The identity \eqref{e21} is just the ribbon
  condition \eqref{e11}.  Thus we have a map $i\zzzcolon \Ob(\rYDHV)\rightarrow Z$, $\hat
  X\mapsto i(\hat X)=(X,X^*,\evv X,\alpha _X,\beta _X)$.  
  
  Conversely, we construct a map $j\zzzcolon Z\rightarrow \Ob(\rYDHV)$ as follows.
  Let $\check X=(X,X^*,\evv X,\alpha _X,\beta _X)\in Z$.
  We have a Yetter--Drinfeld module $(X,\alpha _X,\beta _X)$.  Setting
\begin{gather}
    \label{e29}
 \alpha _{X^*}= \quad {\raisebox{-10mm}{\begin{overpic}[height=20mm, angle=0]{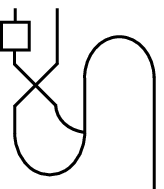}\put(0,60){$H$}\put(15,60){$X^*$}\put(40,-9){$X^*$}\put(2,46){\tiny $S$}\put(29,18){$\alpha_X$}\end{overpic}}} \quad, \quad\quad 
\beta _{X^*}= \quad {\raisebox{-10mm}{\begin{overpic}[height=20mm, angle=0]{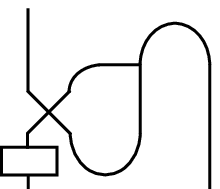}\put(3,-9){$H$}\put(3,60){$X^*$}\put(58,-9){$X^*$}\put(2,7){\tiny $S^{-1}$}\put(45,35){$\beta_X$}\end{overpic}}} \quad,
  \end{gather} \\
  we obtain another Yetter--Drinfeld module $(X^*,\alpha _{X^*},\beta _{X^*})$, since \\
  \begin{gather*}
  \quad {\raisebox{-15mm}{\begin{overpic}[height=25mm, angle=0]{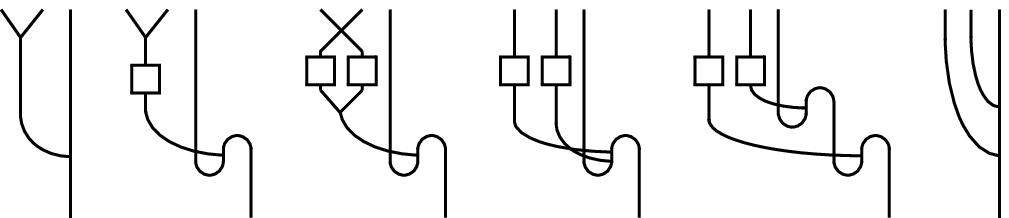}\put(-4,73){$H$}\put(9,73){$H$}\put(20,73){$X^*$}\put(20,-9){$X^*$}\put(25,20){$\alpha_{X^*}$}\put(30,35){$=$}\put(305,73){$H$}\put(318,73){$H$}\put(329,73){$X^*$}\put(325,-9){$X^*$}\put(90,35){$=$}\put(150,35){$=$}\put(216,35){$=$}\put(300,35){$=$}\put(44,44){$S$}\put(102,46){$S$}\put(117,46){$S$}\put(166,46){$S$}\put(180,46){$S$}\put(230,46){$S$}\put(244,46){$S$}\put(332,20){$\alpha_{X^*}$}\put(332,37){$\alpha_{X^*}$}\end{overpic}}} \quad
  \end{gather*} \\
and we obtain the other left $H$-module and $H$-comodule axioms similarly. 
  Using \\
   \begin{gather*}
  \quad {\raisebox{-8mm}{\begin{overpic}[height=20mm, angle=0]{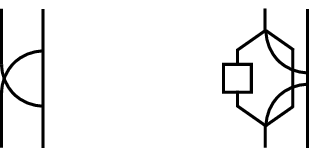}\put(-4,58){$H$}\put(-4,-9){$H$}\put(15,58){$X$}\put(15,-9){$X$}\put(20,15){$\alpha_X$}\put(20,37){$\beta_X$}\put(53,26){$=$}\put(100,58){$H$}\put(117,58){$X$}\put(117,-9){$X$}\put(100,-9){$H$}\put(89,26){$S$}\put(123,32){$\alpha_X$}\put(123,20){$\beta_X$}\end{overpic}}} \quad \quad,
  \end{gather*} \\
 we have \\
   \begin{gather*}
{\raisebox{-10mm}{\begin{overpic}[height=22mm, angle=0]{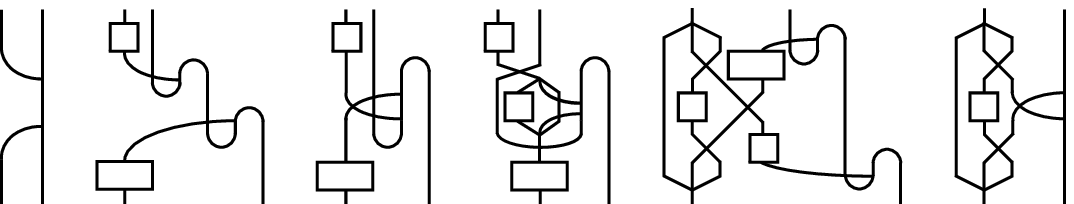}\put(-4,64){$H$}\put(-4,-9){$H$}\put(10,64){$X^*$}\put(10,-9){$X^*$}\put(20,30){$=$}\put(35,50){$S$}\put(30,7){$S^{-1}$}\put(88,30){$=$}\put(104,50){$S$}\put(98,6){$S^{-1}$}\put(138,30){$=$}\put(158,28){$S$}\put(152,50){$S$}\put(159,6){$S^{-1}$}\put(196,30){$=$}\put(211,28){$S$}\put(234,16){$S$}\put(227,40){$S^{-1}$}\put(284,30){$=$}\put(302,64){$H$}\put(327,64){$X^*$}\put(326,-9){$X^*$}\put(301,-9){$H$}\put(303,28){$S$}\put(303,28){$S$}\end{overpic}}} \quad.   
\end{gather*} \\
  Then one can check that $\ev_X,\coev_X,\ev_{X^*},\coev_{X^*}$ are
  morphisms of Yetter--Drinfeld modules as follows.
  The morphism $\ev_X\zzzcolon X^*\otimes X\rightarrow I$ is a morphism of Yetter--Drinfeld modules since \\
  \begin{gather*}
{\raisebox{-10mm}{\begin{overpic}[height=18mm, angle=0]{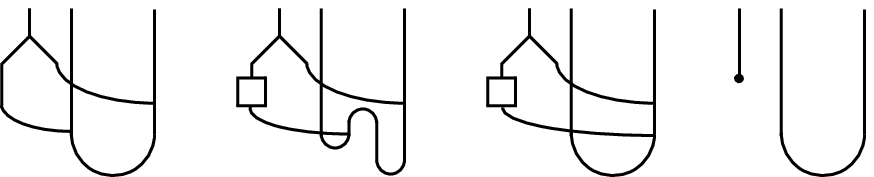}\put(4,54){$H$}\put(18,54){$X^*$}\put(40,54){$X$}\put(71,23){$S$}\put(25,12){$\alpha_{X^*}$}\put(49,21){$\alpha_{X}$}\put(200,30){$=$}\put(144,23){$S$}\put(57,30){$=$}\put(125,30){$=$}\put(213,54){$H$}\put(252,54){$X$}\put(228,54){$X^*$}\end{overpic}}} \quad 
  \end{gather*} \\
  and \\
  \begin{gather*}
{\raisebox{-10mm}{\begin{overpic}[height=21mm, angle=0]{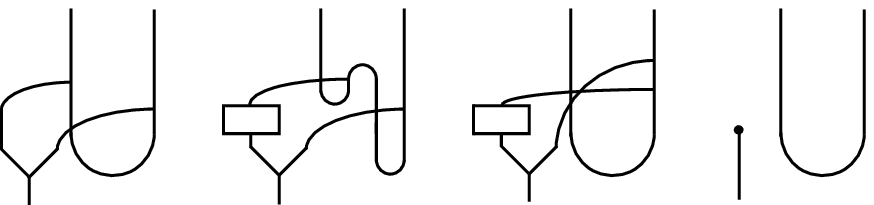}\put(4,-9){$H$}\put(18,62){$X^*$}\put(40,62){$X$}\put(66,23){$S^{-1}$}\put(24,35){$\beta_{X^*}$}\put(48,21){$\beta_{X}$}\put(200,30){$=$}\put(140,23){$S^{-1}$}\put(57,30){$=$}\put(125,30){$=$}\put(214,-9){$H$}\put(252,62){$X$}\put(228,62){$X^*$}\end{overpic}}} \quad .
  \end{gather*} \\
  The morphism $\ev_{X^*}\zzzcolon X\otimes X^*\rightarrow I$ is a morphism of Yetter--Drinfeld modules since \\
  \begin{gather*}
{\raisebox{-10mm}{\begin{overpic}[height=25mm, angle=0]{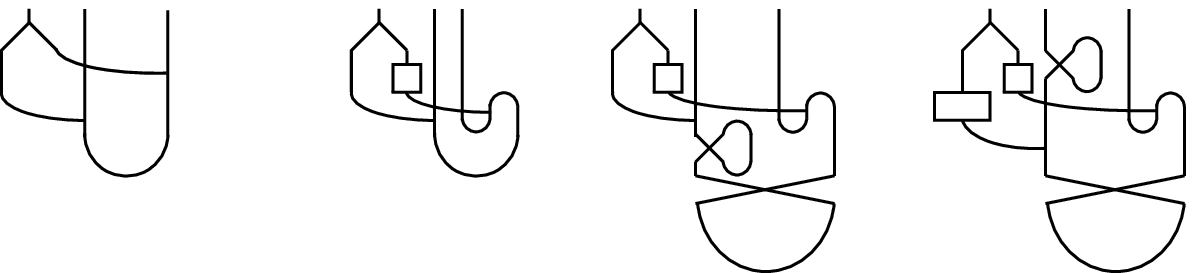}\put(4,75){$H$}\put(18,75){$X$}\put(40,75){$X^*$}\put(104,50){\tiny $S$}\put(25,40){$\alpha_{X}$}\put(48,50){$\alpha_{X^*}$}\put(70,40){$=$}\put(173,50){\tiny $S$}\put(145,40){$=$}\put(230,40){$=$}\put(250,42){\tiny $S^2$}\put(264,50){\tiny $S$} \end{overpic}}} \\
\newline \nonumber \\
{\raisebox{-10mm}{\begin{overpic}[height=25mm, angle=0]{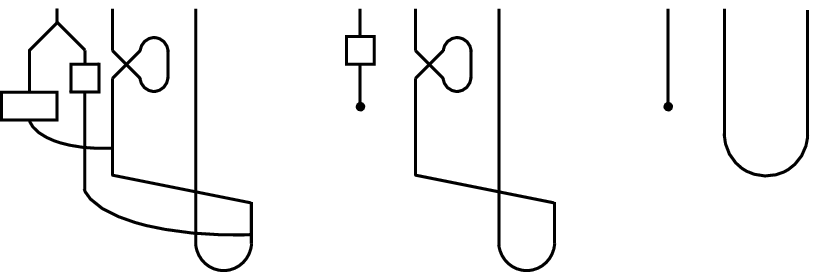}\put(-20,40){$=$}\put(4,43){\tiny $S^2$}\put(20,50){\tiny  $S$}\put(75,40){$=$}\put(93,58){\tiny $S$}\put(155,40){$=$}\put(172,75){$H$}\put(211,75){$X^*$}\put(187,75){$X$}\end{overpic}}} \quad
  \end{gather*} \\
  and similarly we have  \\
    \begin{gather*}
{\raisebox{-10mm}{\begin{overpic}[height=21mm, angle=0]{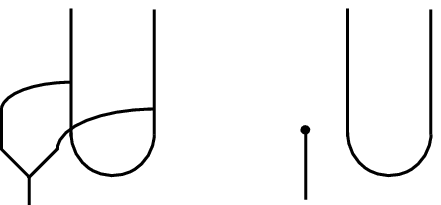}\put(4,-9){$H$}\put(18,62){$X$}\put(40,62){$X^*$}\put(24,35){$\beta_{X}$}\put(48,21){$\beta_{X^*}$}\put(70,30){$=$}\put(86,-9){$H$}\put(100,62){$X$}\put(123,62){$X^*$}\end{overpic}}} \quad .
  \end{gather*} \\
  
  By duality, it follows that $\coev_X$ and $\coev_{X^*}$ are
  morphisms of Yetter--Drinfeld modules, too.  Then we set $j(\check
  X):=((X,\alpha _X,\beta _X),(X^*,\alpha _{X^*},\beta _{X^*}),\ev_X,\coev_X,\ev_{X^*},\coev_{X^*})\in \Ob(\rYDHV)$.

  It is clear that $ij=\id_Z$.  Let us check that
  $ji=\id_{\Ob(\rYDHV)}$.  For
  $\hat{X}=((X,\alpha _X,\beta _X),(X^*,\alpha '_{X^*},\beta '_{X^*}),\evv{X})\in \Ob(\rYDHV)$,
  we have
  \begin{gather*}
    ji(\hat{X})=((X,\alpha _X,\beta _X),(X^*,\alpha _{X^*},\beta _{X^*}),\evv{X}),
  \end{gather*}
  where $\alpha _{X^*}\zzzcolon H\otimes X^*\rightarrow X^*$ and $\beta _{X^*}\zzzcolon X^*\rightarrow H\otimes X^*$ are the
  maps defined in \eqref{e29}.  We need
  to show that $\alpha '_{X^*}=\alpha _{X^*}$ and $\beta '_{X^*}=\beta _{X^*}$.  We have
\vskip1em
  \begin{gather*}
    \alpha '_{X^*}= \quad {\raisebox{-10mm}{\begin{overpic}[height=20mm, angle=0]{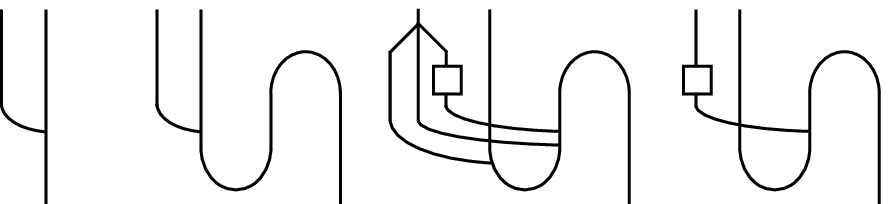}\put(-4,60){$H$}\put(10,60){$X^*$}\put(10,-9){$X^*$}\put(123,33){$S$}\put(161,23){$\alpha_{X}$}\put(180,30){$=$}\put(161,15){$\alpha_{X}$}\put(194,33){$S$}\put(15,21){$\alpha'_{X^*}$}\put(33,30){$=$}\put(100,30){$=$}\put(192,60){$H$}\put(206,60){$X^*$}\put(242,-9){$X^*$}\end{overpic}}} \quad = \alpha _{X^*}.
  \end{gather*} \\
  Similarly, we can check $\beta '_{X^*}=\beta _{X^*}$.
\end{proof}

\begin{remark}
  \label{r23}
  Here we remark the following weaker version of ribbon Yetter--Drinfeld modules.  A
  {\em pivotal Yetter--Drinfeld module} over $H$ in $\modV $ is a pivotal object in
  $\YDHV$.  The category of pivotal Yetter--Drinfeld modules over $H$ in $\modV $ is
  $\pYDHV:=\YDHVp$, the strict braided pivotal category of pivotal
  objects in $\YDHV$.  A pivotal Yetter--Drinfeld module is equivalent to a pair
  of a Yetter--Drinfeld module $(X,\alpha ,\beta )$ over $H$ in $\modV $ and a pivotal object
  structure $(X,X^*,\evv X)$ for $X$ in $\modV $ satisfying \eqref{e19} and
  \eqref{e20}.
\end{remark}

\section{Examples}

In this section, we recall the definition of the tangle category, and give some examples of ribbon Yetter--Drinfeld modules.

\subsection{The tangle category $\modT $}

Here we briefly describe the structure of the category $\modT $ of (framed, oriented)
tangles.  See \cite{Kassel} for details.

The objects in $\modT$ are the words in vertical arrows
$\downarrow ,\uparrow $.  
The monoidal unit is the empty word.  The tensor product
of objects is concatenation of words.
The morphisms $T\zzzcolon w\rightarrow w'$ between two words $w$ to $w'$ are the isotopy classes
of framed, oriented tangles in a cube such that the endpoints with orientations at the
top and bottom  are described by the words $w$ and $w'$.
Composition of two morphisms is composing tangles vertically.
Tensor product on morphisms is placing tangles side by side.
These definitions give a monoidal category structure on $\modT$.

The duality functor $(-)^*:\modT^{\op}\to\modT$ inverts the orientation of arrows and reverses the words.
The duality $d_w$, $b_w$ and the braiding $\psi_{w,w'}$ are defined as in previous sections.

It is well known that $\modT $ is a free strict ribbon category generated
by one object $\downarrow $ \cite{Shum,Turaev}.
Therefore, given a ribbon category $\modC$ and its object $X$, there is a canonical ribbon functor $F:\modT\to \modC$ which takes $\downarrow$ to $X$, which gives rise to a tangle invariant.

\subsection{The Borel subalgebra of $U_h(sl_2)$}

It is well known that for a finite-dimensional Hopf algebra $H$, the category of Yetter--Drinfeld modules is equivalent to the category of left modules over the Drinfeld double $D(H)$ of $H$ \cite{Majid}.
For general Hopf algebra $H$, which may be infinite-dimensional, the definition of the Drinfeld double $D(H)$ does not work directly.
Categories of Yetter--Drinfeld modules work also for infinite-dimensional Hopf algebras. In this subsection, we consider the case where $H=B_h$ is the Borel subalgebra of the quantized enveloping algebra $U_h(sl_2)$, which is an infinite-dimensional complete Hopf algebra over $\mathbb Q[[h]]$.

\subsubsection{The quantized enveloping algebra $U_h(sl_2)$ and its Borel subalgebra}
For the definitions and properties of quantized enveloping algebras, the reader is referred to \cite{Kassel}.

The quantized enveloping algebra $U_h=U_h(sl_2)$ for $sl_2$ is the complete $\mathbb Q[[h]]$-algebra generated by $H,E,F$ with relations
\begin{gather*}
    [H,E]=2E,\quad [H,F]=-2F,\quad
    [E,F]=\frac{K-K^{-1}}{v-v^{-1}},
\end{gather*}
where we set $v=\exp\frac h2$ and $K=v^H=\exp\frac{hH}2$.
(Note that we are working with the symmetric monoidal category of topologically free $\mathbb Q[[h]]$-module.)
It has a complete Hopf algebra structure (with invertible antipode) given by
\begin{gather*}
    \Delta(H)=H\otimes 1+1\otimes H,\quad
    \Delta(E)=E\otimes 1+K\otimes E,\quad
    \Delta(F)=F\otimes K^{-1}+1\otimes F,\quad
    \\
    \epsilon(H)=\epsilon(E)=\epsilon(F)=0,
    \\
    S(H)=-H,\quad S(E)=-K^{-1}E,\quad S(F)=-FK.
\end{gather*}

The Hopf algebra $U_h$ is a complete ribbon Hopf algebra with a universal $R$-matrix
\begin{gather*}
    R= D\left(\sum_{n\ge0}v^{n(n-1)/2}\frac{(v-v^{-1})^n}{[n]!}F^n\otimes E^n\right)
    =D(1\ot1+ (v-v^{-1})F\ot E + \cdots ),
\end{gather*}
where $D=\exp(\frac{h}{4}H\otimes H)\in U_h^{\ho2}$,
and a ribbon element $$\mathbf{r}=\sum S(R')K^{-1}R''=\sum R''KR',$$
where we write $R=\sum R'\otimes R''$.
The associated grouplike element $\kappa\in U_h$ is
\begin{gather*}
    \kappa:=\left(\sum S(R'')R'\right)\mathbf{r}^{-1} = K^{-1}.
\end{gather*}

The Borel subalgebra $B_h$ of $U_h$ is the closed
$\mathbb{Q}[[h]]$-subalgebra of $U_h$ generated by $H$ and $E$.
The algebra $B_h$ is a complete Hopf subalgebra of $U_h$.

\subsubsection{Yetter--Drinfeld modules corresponding to finite-dimensional irreducible $U_h$-modules}
The category $\YD_{B_h}^{fd}$ of finite-dimensional Yetter--Drinfeld modules over $B_h$ is a rigid braided category.
We define a pivotal structure on it by setting
\begin{gather*}
d_{V^*}:V\otimes V^*\to \Qh,\quad d_{V^*}(x\otimes f) = f(\kappa x),\\
b_{V^*}:\Qh\to V^*\otimes V,\quad b_{V^*}(1) = \sum_{i=1}^2 \bfv^i\otimes\kappa^{-1}\bfv_i.
\end{gather*}

It is well known that for each integer $n\ge1$ there is exactly one  ``$n$-dimensional'' $U_h$-module $V_n$ up to isomorphism.  Here ``$n$-dimensional'' means ``free of rank $n$ over $\Qh$''.
Each $V_n$ has a ribbon Yetter--Drinfeld module structure.
For simplicity, we will just look at the $2$-dimensional $U_h$-module $V=V_2=\Qh \bfv_0\oplus \Qh \bfv_1$
defined by
\begin{gather*}
    H\bfv_0= \bfv_0,\quad E\bfv_0=0,\quad F\bfv_0=\bfv_1,\quad
    H\bfv_1= -\bfv_1,\quad E\bfv_1=\bfv_0,\quad F\bfv_1=0.
\end{gather*}
Note that the quantum link invariants associated to the $U_h$-module $V$ given by the Reshetikhin--Turaev construction \cite{RT} is the Jones polynomial.

The left $U_h$-action on $V$ gives a left action $\alpha:B_h\ot V\to V$ of $B_h$ on $V$.
The left coaction $\beta:V\to B_h\otimes V$ is given by
\begin{gather*}
    \beta(v)= \sum R'' \otimes R' v
\end{gather*}
for $v\in V$, where we write $R=\sum R'\otimes R''$.
One computes easily that
\begin{gather*}
    \beta(\bfv_0)= K^{1/2}\otimes \bfv_0 + (v-v^{-1})K^{-1/2}E\ot\bfv_1,\quad
    \beta(\bfv_1)= K^{-1/2}\ot \bfv_1.
\end{gather*}
Then $(V,\alpha,\beta)$ is a Yetter--Drinfeld module over $B_h$,
which determines the Yetter--Drinfeld module structure on $V^*$.
Then one can check that $V$ has a ribbon Yetter--Drinfeld module structure such that the linear isomorphism $\gamma_V: V\to V$ is given by
$\gamma_V(\bfv_i)=\kappa \bfv_i=v^{1-2i}\bfv$.

Note that the braiding $\psi_{V,V}:V\otimes V\to V\otimes V$ constructed from the ribbon Yetter--Drinfeld module structure of $V$ coincides with the braiding given by the universal $R$-matrix.
One can check that the framed link invariant associated to the ribbon Yetter--Drinfeld module $V$ over $B_h$ coincides with the Jones polynomial.

\subsection{Group algebras}

Let $k$ be a field and let $G$ be a finite group.
Let $k[G]$ denote the group Hopf algebra of $G$ over $k$.
The comultiplication, counit and antipode of $k[G]$ are given by
\begin{gather*}
\Delta(g)=g\ot g,\quad
\epsilon(g)=1,\quad
S(g)=g^{-1}
\end{gather*}
for $g\in G$.

We give $k[G]$ a Yetter--Drinfeld module structure over $k[G]$ as follows.
\begin{gather*}
\alpha=\ad:k[G]\ot k[G]\to k[G],\quad \ad(g\ot h) = {}^gh=ghg^{-1},\\
\beta=\Delta: k[G]\to k[G]\ot k[G],\quad \Delta(g) = g\ot g.
\end{gather*}
Thus, the action of $k[G]$ on $k[G]$ is the left adjoint action, and the coaction of $k[G]$ on $k[G]$ is the comultiplication.
Then one can easily check that $(k[G],\ad,\Delta)$ is a Yetter--Drinfeld module over $k[G]$.

Let $S\subset G$ be an $\ad$-invariant subset of $G$, which is a union of conjugacy classes of $G$.
Then the subspace $k[S]=\bigoplus_{s\in S}ks\subset k[G]$ is a Yetter--Drinfeld submodule of $k[G]$.

The Yetter--Drinfeld module $k[S]$ and its dual $k[S]^*=\Hom_k(k[S],k)$ have a ribbon Yetter--Drinfeld module structure as follows.
\begin{gather*}
d_{k[S]}:k[S]^*\otimes k[S]\to k,\quad d_{k[S]}(f\otimes x) = f(x),\\
b_{k[S]}:k\to k[S]\otimes k[S]^*,\quad b_V(1) = \sum_{g\in S} g\otimes g^*,\\
d_{k[S]^*}:k[S]\otimes k[S]^*\to k,\quad d_{k[S]^*}(x\otimes f) = f(x),\\
b_{k[S]^*}:k\to k[S]^*\otimes k[S],\quad b_{k[S]^*}(1) = \sum_{g\in S} g^*\otimes g,
\end{gather*}
where $g^*\in k[S]^*$ is defined by $g^*(h)=\delta_{g,h}$ for $h\in S$.
Then one can check that these maps give a ribbon Yetter--Drinfeld module structure.

One the link invariant corresponding to the ribbon Yetter--Drinfeld module $k[S]$ sends a link $L$ to the number of homomorphisms $\pi_1(\R^3\setminus L)\to G$ such that the meridian of each component of $L$ is sent to an element of $S$.
This invariant is considered by Freyd and Yetter \cite{FY}.

\end{document}